\documentclass[12pt]{amsart}
\usepackage{verbatim,amsmath,amssymb,amsthm}

\newtheorem{lemma}[equation]{Lemma}
\newtheorem{theorem}[equation]{Theorem}
\newtheorem{corollary}[equation]{Corollary}
\newtheorem{proposition}[equation]{Proposition}

\theoremstyle{remark}
\newtheorem*{remark}{Remark}
\newtheorem*{example}{Example}

\renewcommand{\hat}{\widehat}

\DeclareMathOperator\PSL{PSL}
\DeclareMathOperator\im{im}
\newcommand\F{\mathbb{F}}

\begin{document}

\title[Classes of permutation polynomials based on cyclotomy]{Classes of permutation
 polynomials based on cyclotomy and an additive analogue}

\author{Michael E. Zieve}
\address{
Michael E. Zieve\\
Rutgers University\\
Department of Mathematics\\
110 Frelinghuysen Road\\
Piscataway, NJ 08854--8019\\
USA}

\email{zieve@math.rutgers.edu}

\urladdr{www.math.rutgers.edu/$\sim$zieve}

\date{\today}


\thanks{I thank Jos\'e Marcos for sending me preliminary versions of
his paper \cite{M}, and for encouraging me to develop consequences
of his ideas while his paper was still under review.}

\begin{abstract}
I present a construction of permutation polynomials based on cyclotomy,
an additive analogue of this construction, and a generalization of this
additive analogue which appears to have no multiplicative analogue.
These constructions generalize recent results of Jos\'e Marcos.
\end{abstract}

\maketitle

\centerline{\textit{\normalsize Dedicated to Mel Nathanson
    on the occasion of his sixtieth birthday}}
\vspace{\baselineskip}

\section{Introduction}

Writing $\F_q$ for the field with $q$ elements, we consider
\emph{permutation polynomials} over $\F_q$, namely polynomials
$f\in\F_q[x]$ for which the map $\alpha\mapsto f(\alpha)$ induces a
permutation of $\F_q$.  These polynomials first arose in work of
Betti~\cite{Betti}, Mathieu~\cite{Mathieu}, and Hermite~\cite{Hermite},
as a tool for representing and studying permutations.

Since every mapping $\F_q\to\F_q$ is induced by a polynomial, the study
of permutation polynomials focuses on polynomials with unusual properties
beyond inducing a permutation.  In particular, permutation polynomials of
`nice' shapes have been a topic of interest since the work of Hermite,
in which he noted that there are many permutation polynomials of the form
\[
f(x) := a x^i (x^{\frac{q-1}2}+1) - b x^j (x^{\frac{q-1}2}-1)
\]
with $q$ odd, $i,j>0$, and $a,b\in\F_q^*$.  The reason for this is that
$f(\alpha)=2a\alpha^i$ if $\alpha\in\F_q$ is a square, and
$f(\alpha)=2b\alpha^j$ otherwise; thus, for instance, $f$ is a permutation
polynomial if $2a$ and $2b$ are squares and $\gcd(ij,q-1)=1$.

More generally, any polynomial of the form $f(x):=x^r h(x^{(q-1)/d})$ induces
a mapping on $\F_q$ modulo $d$-th powers, so testing whether $f$ permutes
$\F_q$ reduces to testing whether the induced mapping on cosets is bijective
(assuming that $f$ is injective on each coset, or equivalently that $\gcd(r,(q-1)/d)=1$).
The vast majority of known examples of `nice' permutation polynomials have
this `cyclotomic' form for some $d<q-1$; see for instance
[1--5, 7, 9--25, 29--34, 36--43].
Moreover, there is a much longer list of papers proving nonexistence of permutation
polynomials of certain shapes, and nearly all such papers again address these
polynomials $f(x)$ having cyclotomic behavior.

In the recent preprint \cite{M}, Marcos gives five constructions of permutation
polynomials.  His first two constructions are new classes of permutation polynomials
having the above cyclotomic form.  His third construction is a kind of
additive analogue of the first, resulting in polynomials of the form
$L(x)+h(T(x))$ where $T(x):=x^{q/p}+x^{q/p^2}+\dots+x$ is the trace polynomial
from $\F_q$ to its prime field $\F_p$, and $L(x)=\sum a_i x^{p^i}$ is any additive polynomial.
The idea of the analogy is that $T(x)$ induces a homomorphism $\F_q\to\F_p$, just as
$x^{(q-1)/d}$ induces a homomorphism from $\F_q^*$ to its subgroup of $d$-th roots
of unity.  The fourth construction in \cite{M} is a variant of the third for
polynomials of the form $L(x)+h(T(x))(L(x)+c)$,
and the fifth construction replaces $T(x)$ with other symmetric functions in
$x^{q/p}$, $x^{q/p^2}$, \dots, $x$.

In this paper I present rather more general versions of the first four constructions
from \cite{M}, together with simplified proofs.  I can say nothing new about the
fifth construction from \cite{M}, although that construction is quite interesting
and I encourage the interested reader to look into it.


\section{Permutation polynomials from cyclotomy}

In this section we prove the following result, where for $d\ge 1$ we write
$h_d(x):=x^{d-1}+x^{d-2}+\dots+x+1$.

\begin{theorem} \label{one}
Fix a divisor $d>2$ of $q-1$, integers $u\ge 1$ and $k\ge 0$, an element
$b\in\F_q$, and a polynomial $g\in\F_q[x]$ divisible by $h_d$.  Then
\[
f(x) := x^u \left(b x^{k(q-1)/d} + g(x^{(q-1)/d})\right)
\]
permutes\/ $\F_q$ if and only if the following four conditions hold:
\begin{enumerate}
\item $\gcd(u,(q-1)/d)=1$,
\item $\gcd(d, u+k(q-1)/d)=1$,
\item $b\ne 0$,
\item $1+g(1)/b$ is a $d$-th power in $\F_q^*$.
\end{enumerate}
\end{theorem}

The proof uses the following simple lemma.

\begin{lemma} \label{lem}
Fix a divisor $d$ of $q-1$, an integer $u>0$, and a polynomial $h\in\F_q[x]$.
Then $f(x):=x^u h(x^{(q-1)/d})$ permutes\/ $\F_q$ if and only if the following
two conditions hold:
\begin{enumerate}
\item $\gcd(u,(q-1)/d)=1$,
\item $\hat f(x):=x^u h(x)^{(q-1)/d}$ permutes the set $\mu_d$ of $d$-th roots of unity~in~$\F_q^*$.
\end{enumerate}
\end{lemma}

I discovered this lemma in 1997 when writing \cite{TZ}, and used it in seminars
and private correspondence, but I did not publish it
until recently \cite[Lemma~2.1]{Z1}.  For other applications
of this lemma, see \cite{MZ,Z1,Z2}.

\begin{proof}[Proof of Theorem~\ref{one}]
In light of the lemma, we just need to determine when $\hat f(x)$ permutes
$\mu_d$, where
\[
\hat f(x) := x^u (b x^k + g(x))^{(q-1)/d}.
\]
For $\zeta\in\mu_d\setminus\{1\}$ we have $g(\zeta)=0$, so
$\hat f(\zeta)=b^{(q-1)/d} \zeta^{u+k(q-1)/d}$.  Thus, $\hat f$ is injective
on $\mu_d\setminus\{1\}$ if and only if $b\ne 0$ and $\gcd(d,u+k(q-1)/d)=1$.
When these conditions hold, $\hat f(\mu_d\setminus\{1\})=\mu_d\setminus\{b^{(q-1)/d}\}$,
so $\hat f$ permutes $\mu_d$ if and only if $\hat f(1)=b^{(q-1)/d}$.
Since $\hat f(1)=(b+g(1))^{(q-1)/d}$, the latter condition is equivalent to
$(1+g(1)/b)^{(q-1)/d}=1$, as desired.
\end{proof}

The case $g=h_d$ of Theorem~\ref{one} is \cite[Thm.~2]{M},
and \cite[Prop.~4]{M} is the case that $g=h_5-x^3-x^4$ and $d=u=k-2=5$.

\begin{remark}
The key feature of the polynomials in
Theorem~\ref{one} as a particular case of Lemma~\ref{lem} is that the induced
mapping $\hat f$ on $\mu_d\setminus\{1\}$ is a monomial, and we know when monomials permute
$\mu_d$.  For certain values of $d$, we know other permutations of $\mu_d$:
for instance, if $q=q_0^2$ and $d=\sqrt{q_0}-1$ then $\mu_d=\F_{q_0}^*$, so
we can obtain permutation polynomials over $\F_q$ by applying Lemma~\ref{lem}
to polynomials $f(x)$ for which the induced map $\hat f$ on $\F_{q_0}^*$ is any
prescribed permutation polynomial.  This construction already yields interesting
permutation polynomials of $\F_q$ coming from degree-$3$ permutation polynomials
of $\F_{q_0}$; see \cite{TZ} for details and related results.
\end{remark}


\section{Permutation polynomials from additive cyclotomy}

Lemma~\ref{lem} addresses maps $\F_q\to\F_q$ which respect the partition of
$\F_q^*$ into cosets modulo a certain subgroup.  In this section we give an
analogous result in terms of cosets of the additive group of $\F_q$ modulo
a subgroup.
Let $p$ be the characteristic of $\F_q$.  An \emph{additive} polynomial over
$\F_q$ is a polynomial of the form $\sum_{i=0}^k a_i x^{p^i}$ with $a_i\in\F_q$.
The key property of additive polynomials $A(x)$ is that they induce
homomorphisms on the additive group of $\F_q$, since
$A(\alpha+\beta)=A(\alpha)+A(\beta)$ for $\alpha,\beta\in\F_q$.
The additive analogue of Lemma~\ref{lem} is as follows, where 
we write $\im B$ and $\ker B$ for
the image and kernel of the mapping $B\colon\F_q\to\F_q$.

\begin{proposition} \label{add}
Pick additive $A,B\in\F_q[x]$ and
an arbitrary $g\in\F_q[x]$.  Then $f(x):=A(x)+g(B(x))$ permutes\/ $\F_q$
if and only if $A(\ker B) + \hat f(\im B)=\F_q$, where
$\hat f(x):=g(x)+A(\hat B(x))$ and $\hat B\in\F_q[x]$ is any polynomial for which
$B(\hat B(x))$ is the identity on $\im B$.  In other words, $f$ permutes\/ $\F_q$
if and only if $\hat f$ induces a bijection $\im B \to \F_q/A(\ker B)$, where
$\F_q/A(\ker B)$ is the quotient of the additive group of\/ $\F_q$ by the subgroup
$A(\ker B)$.
\end{proposition}

\begin{proof}
For $\beta\in\ker B$ we have $f(x+\beta)=A(x)+A(\beta)+g(B(x))=f(x)+A(\beta)$.
Thus, for $\alpha\in\F_q$ we have $f(\alpha+\ker B)=f(\alpha)+A(\ker B)$.
Since $\F_q=\ker B + \hat B(\im B)$, it follows that
$f(\F_q)=f(\hat B(\im B))+A(\ker B)$.
Since $f(\hat B(\gamma))=A(\hat B(\gamma))+g(B(\hat B(\gamma)))=A(\hat B(\gamma))+g(\gamma)$
for $\gamma\in\im B$, the result follows.
\end{proof}

\begin{corollary}
If $f$ permutes\/ $\F_q$ then $A$ is injective on $\ker B$ and $\hat f$ is injective
on $\im B$.
\end{corollary}

\begin{proof}
If $A(\ker B)+\hat f(\im B)=\F_q$ then
\[
q\le\#A(\ker B)\cdot\#\hat f(\im B)\le \#(\ker B)\cdot\#(\im B)=q,
\]
where the last equality holds because $B$ defines a homomorphism on the
additive group of $\F_q$.  The result follows.
\end{proof}

\begin{corollary}
Suppose $A(B(\alpha))=B(A(\alpha))$ for all $\alpha\in\F_q$.  Then $f$ permutes\/ $\F_q$
if and only if $A$ permutes $\ker B$ and $A(x)+B(g(x))$ permutes $\im B$.
\end{corollary}

\begin{proof}
Since $A$ and $B$ commute, and $A(0)=0$, it follows that $A(\ker B)\subseteq\ker B$.
Thus, by the previous corollary, if $f$ permutes $\F_q$ then $A$ permutes $\ker B$.
Henceforth assume that $A$ permutes $\ker B$.  By the proposition, $f$ permutes $\F_q$ if and only if
$\ker B + \hat f(\im B)=\F_q$; since the left side is the preimage under $B$ of
$B(\hat f(\im B))$, this condition may be restated as $B(\hat f(\im B))=\im B$.
For $\gamma\in\im B$ we have
$B(\hat f(\gamma))=B(g(\gamma))+B(A(\hat B(\gamma)))=B(g(\gamma))+A(B(\hat B(\gamma)))
 = B(g(\gamma))+A(\gamma)$, so $B(\hat f(x))$ permutes $\im B$ if and only if
$B(g(x))+A(x)$ permutes $\im B$.
\end{proof}

One way to get explicit examples satisfying the conditions of this result is as
follows: if $B=x^{q/p}+x^{q/p^2}+\dots+x^p+x$ and $A\in\F_p[x]$, then $A(B(x))=B(A(x))$,
so $f$ permutes $\F_q$ if and only if $A$ permutes $\ker B$ and $A(x)+B(g(x))$
permutes $\im B=\F_p$.  In case $g$ is a constant (in $\F_q$) times a polynomial
over $\F_p$, this becomes (a slight generalization of) \cite[Thm.~6]{M}.  The following
case of \cite[Cor.~8]{M} exhibits this.

\begin{example}
In case $q=p^2$ and $B=x^p+x$ and $A=x$, the previous corollary says $f(x):=x+g(x^p+x)$
permutes\/ $\F_{p^2}$ if and only if $x+g(x)^p+g(x)$ permutes\/ $\F_p$, which trivially holds
when $g=\gamma h(x)$ with $h\in\F_p[x]$ and $\gamma^{p-1}=-1$.  For instance,
taking $h(x)=x^2$, it follows that $x+\gamma (x^p+x)^2$ permutes\/ $\F_{p^2}$.  By using
other choices of $h$, we can make many permutation polynomials over\/ $\F_{p^2}$ whose degree
is a small multiple of $p$.  This is of interest because heuristics suggest
that `at random' there would be no permutation polynomials over\/ $\F_q$ of
degree less than $q/(2\log q)$.  The bulk of the known low-degree permutation
polynomials are \emph{exceptional}, in the sense that they permute
$\F_{q^k}$ for infinitely many~$k$; a great deal is known about these
exceptional polynomials, for instance see \cite{GZ}.  It is known that
any permutation polynomial of degree at most $q^{1/4}$ is exceptional.
However, the examples described above have degree on the order of
$q^{1/2}$ and are generally not exceptional.
\end{example}

Our final result generalizes the above example in a different direction
than Proposition~\ref{add}.

\begin{theorem} \label{last}
Pick any $g\in\F_q[x]$, any additive $A\in\F_p[x]$, and any $h\in\F_p[x]$.
For $B:=x^{q/p}+x^{q/p^2}+\dots+x^p+x$, the polynomial
$f(x):=g(B(x)) + h(B(x))A(x)$ permutes\/ $\F_q$ if and only if $A$ permutes
$\ker B$ and $B(g(x))+h(x)A(x)$ permutes\/ $\F_p$ and $h$ has no roots in $\F_p$.
\end{theorem}

\begin{proof}
For $\beta\in\ker B$ we have $f(x+\beta)=f(x)+h(B(x))A(\beta)$.
Thus, if $f$ permutes $\F_q$ then $A$ is injective on $\ker B$ and $h$ has no roots in $\F_p$.
Since $A(B(x))=B(A(x))$ and $A(0)=0$, also $A(\ker B)\subseteq\ker B$,
so if $f$ permutes $\F_q$ then $A$ permutes $\ker B$.
Henceforth assume $A$ permutes $\ker B$ and $h$ has no roots in $\F_p$.
Since $\im B=\F_p$ and $h(\F_p)\subseteq\F_p\setminus\{0\}$,
we have $h(B(\alpha))\in\F_p\setminus\{0\}$ for $\alpha\in\F_q$.  Thus,
for $\alpha\in\F_q$ we have $f(\alpha+\ker B) = f(\alpha)+\ker B$, so
$f$ permutes $\F_q$ if and only if $B(f(\F_q))=\im B$.  Now for $\alpha\in\F_q$ we have
$B(f(\alpha))=B(g(B(\alpha)))+B(h(B(\alpha))A(\alpha))$, and since $h(B(\alpha))\in\F_p$ this becomes
$B(f(\alpha))=B(g(B(\alpha)))+h(B(\alpha))B(A(\alpha))=B(g(B(\alpha))+h(B(\alpha))A(B(\alpha))$,
so $B(f(\F_q))$ is the image of $\im B$ under $B(g(x))+h(x)A(x)$.  The result follows.
\end{proof}

In case $g=\gamma h+\delta$ with $\gamma,\delta\in\F_q$, the above result
becomes a generalization of \cite[Thm.~10]{M}.
In view of the analogy between
Lemma~\ref{lem} and Proposition~\ref{add}, it is natural to seek a `multiplicative'
analogue of Theorem~\ref{last}.  However, I have been unable to find such a result:
the obstacle is that the polynomial $f$ in Theorem~\ref{last} is the sum of products
of polynomials, which apparently should correspond to a product of powers of polynomials,
but the latter is already included in Lemma~\ref{lem}.



\begin{thebibliography}{99.}
\newcommand{\au}[1]{{#1}:}
\newcommand{\ti}[1]{{#1}.}
\newcommand{\jo}[1]{{#1}}
\newcommand{\vo}[1]{\textbf{#1},}
\newcommand{\yr}[1]{(#1)}
\newcommand{\ppx}[1]{#1,}
\newcommand{\pp}[1]{#1}
\newcommand{\pps}[1]{#1}
\newcommand{\bk}[1]{{#1},}
\newcommand{\inbk}[1]{In: \bk{#1}}
\newcommand{\xxx}[1]{{arXiv:#1}}

\bibitem{A}
\au{S. Ahmad}
\ti{Split dilations of finite cyclic groups with applications to finite fields}
\jo{Duke Math. J.}
\vo{37}
\pp{547--554}
\yr{1970}

\bibitem{AAW}
\au{A. Akbary, S. Alaric and Q. Wang}
\ti{On some classes of permutation polynomials}
\jo{Int. J. Number Theory}
\vo{4}
\pp{121--133}
\yr{2008}

\bibitem{AW}
\au{A. Akbary and Q. Wang}
\ti{On some permutation polynomials over finite fields}
\jo{Int. J. Math. Math. Sci.}
\vo{16}
\pp{2631--2640}
\yr{2005}

\bibitem{AW2}
\au{A. Akbary and Q. Wang}
\ti{A generalized Lucas sequence and permutation binomials}
\jo{Proc. Amer. Math. Soc.}
\vo{134}
\pp{15--22}
\yr{2006}

\bibitem{AW3}
\au{A. Akbary and Q. Wang}
\ti{On polynomials of the form $x^r f(x^{(q-1)/l})$}
\jo{Int. J. Math. Math. Sci.}
\yr{2007}
art. ID 23408.

\bibitem{Betti}
\au{E. Betti}
\ti{Sopra la risolubilit\'a per radicali delle equazioni algebriche
irriduttibili di grado primo}
\jo{Annali Sci. Mat. Fis.}
\vo{2}
\pps{5--19}
\yr{1851}
[= Op. Mat.
\vo{1}
\pp{17--27}
\yr{1903}]

\bibitem{Brioschi}
\au{F. Brioschi}
\ti{Des substitutions de la forme $\theta(r)\equiv\epsilon(r^{n-2}+ar^{(n-3)/2})$
pour un nombre $n$ premier de lettres}
\jo{Math. Ann.}
\vo{2}
\pps{467--470}
\yr{1870}
[= Op. Mat.
\vo{5}
\pp{193--197}
\yr{1909}]

\bibitem{Ca0}
\au{L. Carlitz}
\ti{Permutations in a finite field}
\jo{Proc. Amer. Math. Soc.}
\vo{4}
\pp{538}
\yr{1953}

\bibitem{Ca}
\au{L. Carlitz}
\ti{Some theorems on permutation polynomials}
\jo{Bull. Amer. Math. Soc.}
\vo{68}
\pp{120--122}
\yr{1962}

\bibitem{Ca2}
\au{L. Carlitz}
\ti{Permutations in finite fields}
\jo{Acta Sci. Math. (Szeged)}
\vo{24}
\pp{196--203}
\yr{1963}

\bibitem{CW}
\au{L. Carlitz and C. Wells}
\ti{The number of solutions of a special system of equations in a finite field}
\jo{Acta Arith.}
\vo{12}
\pp{77--84}
\yr{1966}

\bibitem{CM}
\au{S.~D. Cohen and R.~W. Matthews}
\ti{A class of exceptional polynomials}
\jo{Trans. Amer. Math. Soc.}
\vo{345}
\pp{897--909}
\yr{1994}

\bibitem{CM2}
\au{S.~D. Cohen and R.~W. Matthews}
\ti{Exceptional polynomials over finite fields}
\jo{Finite Fields Appl.}
\vo{1}
\pp{261--277}
\yr{1995}

\bibitem{Dicksonthesis}
\au{L.~E. Dickson}
\ti{The analytic representation of substitutions on a power of a prime
number of letters with a discussion of the linear group}
\jo{Ann. of Math.}
\vo{11}
\pp{65--120}
\yr{1896}

\bibitem{Dickson}
\au{L.~E. Dickson}
\bk{Linear Groups with an Exposition of the Galois Field Theory}
Teubner, Leipzig
\yr{1901}
[Reprinted by Dover, New York (1958)]

\bibitem{E}
\au{A.~B. Evans}
\bk{Orthomorphism Graphs of Groups}
Springer-Verlag, Heidelberg (1992)

\bibitem{E2}
\au{A.~B. Evans}
\ti{Cyclotomy and orthomorphisms: a survey}
\jo{Congr. Numer.}
\vo{101}
\pp{97--107}
\yr{1994}

\bibitem{Fillmore}
\au{J.~P. Fillmore}
\ti{A note on split dilations defined by higher residues}
\jo{Proc. Amer. Math. Soc.}
\vo{18}
\pp{171--174}
\yr{1967}

\bibitem{GZ}
\au{R.~M. Guralnick and M.~E. Zieve}
\ti{Polynomials with $\PSL(2)$ monodromy}
\jo{Annals of Math.},
to appear,
\xxx{0707.1835}

\bibitem{Hermite}
\au{Ch. Hermite}
\ti{Sur les fonctions de sept lettres}
\ti{C. R. Acad. Sci. Paris}
\vo{57}
\pps{750--757}
\yr{1863}
[= Ouvres \vo{2} 280--288 (1908)]

\bibitem{JL}
\au{N.~S. James and R. Lidl}
\ti{Permutation polynomials on matrices}
\jo{Linear Algebra Appl.}
\vo{96}
\pp{181--190}
\yr{1987}

\bibitem{KL}
\au{S.~Y. Kim and J.~B. Lee}
\ti{Permutation polynomials of the type $x^{1+(q-1)/m}+ax$}
\jo{Comm. Korean Math. Soc.}
\vo{10}
\pp{823--829}
\yr{1995}

\bibitem{LC}
\au{Y. Laigle-Chapuy}
\ti{Permutation polynomials and applications to coding theory}
\jo{Finite Fields Appl.}
\vo{13}
\pp{58--70}
\yr{2007}

\bibitem{LP}
\au{J.~B. Lee and Y.~H. Park}
\ti{Some permuting trinomials over finite fields}
\jo{Acta Math. Sci.}
\vo{17}
\pp{250--254}
\yr{1997}

\bibitem{LZ}
\au{H.~W. Lenstra, Jr. and M. Zieve}
\ti{A family of exceptional polynomials in characteristic three}
\inbk{Finite Fields and Applications}
Cambridge Univ. Press, Cambridge
\pp{209--218}
\yr{1996}

\bibitem{M}
\au{J.~E. Marcos}
\ti{Specific permutation polynomials over finite fields}
submitted for publication,
\xxx{0810.2738v1}

\bibitem{MZ}
\au{A.~M. Masuda and M.~E. Zieve}
\ti{Permutation binomials over finite fields}
\jo{Trans. Amer. Math. Soc.},
to appear,
\xxx{0707.1108}

\bibitem{Mathieu}
\au{E. Mathieu}
\ti{M\'emoire sur l'\'etude des fonctions de plusieurs quantit\'es sur la
mani\`ere de les former, et sur les substitutions qui les laissent
invariables}
\jo{J. Math. Pures Appl.}
\vo{6}
\pp{241--323}
\yr{1861}

\bibitem{MN}
\au{G. Mullen and H. Niederreiter}
\ti{The structure of a group of permutation polynomials}
\jo{J. Austral. Math. Soc. (Ser. A)}
\vo{38}
\pp{164--170}
\yr{1985}

\bibitem{NR}
\au{H. Niederreiter and K.~H. Robinson}
\ti{Complete mappings of finite fields}
\jo{J. Austral. Math. Soc. (Ser. A)}
\vo{33}
\pp{197--212}
\yr{1982}

\bibitem{NW}
\au{H. Niederreiter and A. Winterhof}
\ti{Cyclotomic $\mathcal R$-orthomorphisms of finite fields}
\jo{Discrete Math.}
\vo{295}
\pp{161--171}
\yr{2005}

\bibitem{PL0}
\au{Y.~H. Park and J.~B. Lee}
\ti{Permutation polynomials with exponents in an arithmetic progression}
\jo{Bull. Austral. Math. Soc.}
\vo{57}
\pp{243--252}
\yr{1998}

\bibitem{PL}
\au{Y.~H. Park and J.~B. Lee}
\ti{Permutation polynomials and group permutation polynomials}
\jo{Bull. Austral. Math. Soc.}
\vo{63}
\pp{67--74}
\yr{2001}

\bibitem{Rogers}
\au{L.~J. Rogers}
\ti{On the analytical representation of heptagrams}
\jo{Proc. London Math. Soc.}
\vo{22}
\pp{37--52}
\yr{1890}

\bibitem{TZ}
\au{T.~J. Tucker and M.~E. Zieve}
\ti{Permutation polynomials, curves without points, and Latin squares}
preprint (2000)

\bibitem{W0}
\au{D. Wan}
\ti{Permutation polynomials over finite fields}
\jo{Acta Math. Sinica}
\vo{3}
\pp{1--5}
\yr{1987}

\bibitem{W}
\au{D. Wan}
\ti{Permutation binomials over finite fields}
\jo{Acta Math. Sinica}
\vo{10}
\pp{30--35}
\yr{1994}

\bibitem{WL}
\au{D. Wan and R. Lidl}
\ti{Permutation polynomials of the form $x^r f(x^{(q-1)/d})$ and their
group structure}
\jo{Monatsh. Math.}
\vo{112}
\pp{149--163}
\yr{1991}

\bibitem{LWang}
\au{L. Wang}
\ti{On permutation polynomials}
\jo{Finite Fields Appl.}
\vo{8}
\pp{311--322}
\yr{2002}

\bibitem{Wells1}
\au{C. Wells}
\ti{Groups of permutation polynomials}
\jo{Monatsh. Math.}
\vo{71}
\pp{248--262}
\yr{1967}

\bibitem{Wells2}
\au{C. Wells}
\ti{A generalization of the regular representation of finite abelian groups}
\jo{Monatsh. Math.}
\vo{72}
\pp{152--156}
\yr{1968}

\bibitem{Z1}
\au{M.~E. Zieve}
\ti{Some families of permutation polynomials over finite fields}
\jo{Internat. J. Number Theory}
\vo{4}
\yr{2008}, to appear,
\xxx{0707.1111}

\bibitem{Z2}
\au{M.~E. Zieve}
\ti{On some permutation polynomials over $\F_q$ of the form
$x^r h(x^{(q-1)/d})$}
\jo{Proc. Amer. Math. Soc.},
to appear,
\xxx{0707.1110}

\end{thebibliography}
\end{document}